\documentclass[12pt]{amsart}
\usepackage[active]{srcltx}
\usepackage{a4wide}
\usepackage{amsthm,amsfonts,amsmath,mathrsfs,amssymb}
\usepackage{dsfont}
\usepackage[T1]{fontenc}
\usepackage[utf8]{inputenc}
\usepackage{fourier}
\usepackage{pdfpages}
\usepackage{graphicx}

\def\a{\alpha}       \def\b{\beta}        \def\g{\gamma}
\def\d{\delta}            
       \def\t{\theta}       \def\f{\phi}
         \def\r{\rho}         
               \def\t{\theta}

\def\G{\Gamma}

\def\D{{\mathbb D}}     
\def\C{{\mathbb C}}

\def\({\left(}       \def\){\right)}

\newtheorem{prop}{\sc Proposition}

\newtheorem{thm}[prop]{\sc Theorem}
\newtheorem{cor}[prop]{\sc Corollary}

\newtheorem{other}{\sc Theorem}              

\begin{document}
\title[Contractive inequalities for mixed norm spaces]{Contractive inequalities for mixed norm spaces and the Beta function}
\author[A. Llinares]{Adri\'an Llinares}
\address{Departamento de Matem\'aticas, Universidad Aut\'onoma de
Madrid, 28049 Madrid, Spain}
\email{adrian.llinares@uam.es}
\author[D. Vukoti\'c]{Dragan Vukoti\'c}
\address{Departamento de Matem\'aticas, Universidad Aut\'onoma de
Madrid, 28049 Madrid, Spain} \email{dragan.vukotic@uam.es}
\date{2 December, 2021.}
\begin{abstract}
For a wide range of pairs of mixed norm spaces such that one space is contained in another, we characterize all cases when contractive norm inequalities hold. In particular, this yields such results for many pairs of weighted Bergman spaces. Some inequalities of this type are motivated by their applications in Number Theory and in Mathematical Physics.
\end{abstract}
\maketitle
\section{Introduction}
 \label{sec-intro}

\subsection{Preliminaries}
 \label{subsec-prelim}
Let $0<p$, $q\le \infty$, $0<a<\infty$. The mixed norm space $H(p,q,a)$ is defined as the set of all functions $f$ that are analytic in the unit disk $\D$ and satisfy the condition
\begin{equation}
 \|f\|_{p,q,a} = \( a q \int_0^1 2 \r (1-\r^2)^{aq-1} M_p^q(\r;f) d\r \)^{1/q} <\infty\,, \quad 0<q<\infty\,,
 \label{eq-norm}
\end{equation}
\begin{equation}
 \|f\|_{p,\infty,a} = \sup_{0\le \r<1} (1-\r^2)^{a} M_p(\r;f)<\infty\,,
 \label{eq-infty}
\end{equation}
where
\begin{equation}
 M_p(\r;f) = \( \int_0^{2\pi} |f(\r e^{i\t})|^p \frac{d\t}{2\pi} \)^{1/p}
 \label{eq-int-means}
\end{equation}
denotes the usual integral means of $f$ of order $p$ on the circle $\{z\,:\,|z|=\r\}$. An important special case is the standard weighted Bergman space $A^p_\a = H(p,p,\frac{\a+1}{p})$, $-1<\a<\infty$, with
$$
 \|f\|_{A^p_\a} = \( \int_\D (\a+1) (1-|z|^2)^{\a} |f(z)|^p dA(z)\)^{1/p}<\infty \,,
$$
where $dA(z)=\frac{1}{\pi}\,dx\,dy$ is the normalized Lebesgue area measure on $\D$. The classical Hardy space $H^p$ consisting of all functions analytic in $\D$ for which the finite limit
$$
 \|f\|_{H^p} = \lim_{r\to 1^-} M_p (r;f) = \sup_{0<r<1} M_p (r;f)
$$
exists can be understood as the limit case $H(p,\infty,0)$.
\par
Integrals as in \eqref{eq-norm} had been considered already by Hardy and Littlewood but the space $H(p,q,a)$ was only formally defined and systematically studied later, first by Hardy's student Thomas Flett \cite{F1, F2} and later in \cite{AJ}, \cite{Bl}, among many other papers. Inclusions between different mixed normed spaces in many cases have been known, the information being scattered in the literature. The classification below was completed to include the missing cases in \cite{A, A-th}. In both theorems below, we assume that $0<a$, $b<\infty$ and $0 < p$, $q$, $u$, $v\le \infty$.
\par
\begin{other} \label{thm-ia1}
If $p\ge u$, then $H(p,q,a)\subset H(u,v,b)$ if and only if either $a<b$ or $a=b$ and $q\le v$.
\end{other}
\par
\begin{other} \label{thm-ia2}
If $p<u$, then $H(p,q,a)\subset H(u,v,b)$ if and only if either $a+\frac{1}{p}<b+\frac{1}{u}$ or $a+\frac{1}{p}=b+\frac{1}{u}$ and $q\le v$.
\end{other}
As a corollary for the weighted Bergman spaces, we have the well-known results (see, for example, \cite[Theorem~1.3]{BKV}): when $p\ge q$, the inclusion $A^p_\a\subset A^q_\b\/$ holds only in the trivial case $p=q$ and $\a=\b$ or when $\frac{\a+1}{p} < \frac{\b+1}{q}$. When $p<q$, we have the inclusion $A^p_\a\subset A^q_\b\/$ if and only if \ $\frac{\a+2}{p} < \frac{\b+2}{q}$.

\subsection{Motivation and results}
 \label{subsec-motivation}
An important and natural question is: when does the contractive inequality:
$$
 \|f\|_{u,v,b} \le \|f\|_{p,q,a}
$$
hold for all $f\in H(p,q,a)$? Note that if it holds, it must be sharp since the constant function one has unit norm in all mixed norm spaces. A more specific question is when is the inclusion between two weighted Bergman spaces contractive.
\par\smallskip
As the main result of this note, we prove that, if $p\ge u$, the inclusion $H(p,q,a)\subset H(u,v,b)$ is contractive if and only if either \ (1) $q\le v$ and $a\le b$, \ or \ (2) $q>v$ and $a q \le b v$. \par
As a corollary, whenever $p\ge q$, we obtain that the contractive inequality for weighted Bergman spaces: \ $\|f\|_{A^q_\b} \le \|f\|_{A^p_\a}$ \ holds if and only if $\a\le\b$.
\par\smallskip
Some motivation for this study is in order. Several papers in the literature have been devoted to the study of contractive inequalities under the conditions of Theorem~\ref{thm-ia2}. For example, a very  natural question is whether the following special class of inequalities as above for weighted Bergman norms
\begin{equation}
 \|f\|_{A^p_{cp-2}} \le \|f\|_{A^q_{cq-2}} \,, \qquad \frac{1}{c}<q<p <\infty\,,
 \label{eq:contract-ineq}
\end{equation}
holds for every fixed $f$ analytic in $\D$. This problem turns out to be extremely difficult but its solution would be of interest for several applications. For example, a conjecture of Lieb and Solovej \cite{LS} related to certain inequalities for entropies can be reformulated for the disk as the particular case $c=\a+\frac12$ and $q=2$ of \eqref{eq:contract-ineq}, that is:
\begin{equation}
 \|f\|_{A^p_{\frac{2\a+1}{2}p-2}} \le \|f\|_{A^2_{2\a-1}}\,, \qquad 2\le p<\infty\,, \quad 0<\a<\infty\,.
 \label{eq:Lieb-Solovej}
\end{equation}
This conjecture has so far been  proved only in certain special cases, for example, $p=n/2$, where $n$ is a positive integer. Different statements and proofs can be found in the papers \cite{Bu}, \cite{BBHOP} and in the most recent works \cite{LS} and \cite{BGS}. In the classical Fock (or Bargmann-Segal) spaces $F^p_\a$ of entire functions which are $p$-integrable with respect to a Gaussian measure in the plane with parameter $\a$, there is also a contractive inequality for the spaces $F^p_\a$ and $F^\infty_\a$ (\textit{cf.\/} \cite[p. 40]{Z}). Contractive inequalities for generalized Fock spaces in $\C^n$ were studied in \cite{C}, again in relation to the Wehrl entropy conjecture.
\par
Recently, in \cite{BBHOP} a conjecture was formulated that  \eqref{eq:contract-ineq} should hold for $c=1/2$; this was already implicitly suggested in \cite{BOSZ}. This would have some important consequences. It is well known that, for $0<c<\infty$, the spaces $A^{p}_{cp-2}$ become larger as $p\/$ increases. Also, the Hardy space $H^{1/c}$ can be viewed as the limit case of these spaces as $p\to \frac{1}{c}+$, in the sense that
$$
 \|f\|_{H^{1/c}} = \lim_{p\to 1/c+} \|f\|_{p,\,cp-2}\,.
$$
Thus, if true, inequality \eqref{eq:contract-ineq} would also readily imply the following contractive inequality mentioned in  \cite[Problem~2.1,~p.~53]{Pa}:
\begin{equation}
 \|f\|_{A^p_{cp-2}} \le \|f\|_{H^{1/c}} \,, \qquad \frac{1}{c}<p<\infty
 \label{eq:contract-ineq-w}
\end{equation}
by taking the limit as $q\to 1/c+$. In \cite{BOSZ}, a conjecture was posed that
\begin{equation}
 \|f\|_{A^p_{\frac{p}{2}-2}} \le \|f\|_{H^2}\,, \quad  2<p<\infty\,,
 \label{eq:Seip-et-al}
\end{equation}
which is a special case of \eqref{eq:contract-ineq-w} with $c=1/2$. Important related inequalities have been proved in \cite{BBHOP}. Proving such inequality would have immediate applications to the theory of Hardy spaces of Dirichlet series and, hence, further direct consequences in Number Theory; \textit{cf.\/} \cite{BHS}.
\par\smallskip
In summary, the case $p<u$ (scenario of Theorem~\ref{thm-ia2}) seems to be very difficult even in special cases. While some partial results on these cases will be treated elsewhere, in Section~\ref{sect-res-pfs} this note we devote our effort to characterizing completely the cases when norm inclusions are contractive under the condition of Theorem~\ref{thm-ia1}: $p\ge u$. It is our belief that this case should  also be of some interest and that it seems appropriate to carry out a systematic study of when the inclusions can be contractive in all cases.
\par
We also list some direct consequences that might be of independent interest, such as certain inequalities for the Beta function which we have not been able to find elsewhere in the literature. This is done in Section~\ref{sec-ineq-Beta}.

\section{Contractive norm inequalities in the case $p\ge u$}
 \label{sect-res-pfs}
In what follows, in simplify the notation, for a fixed function $f$ analytic in the unit disk, we shall write throughout
$$
 m_p(r) = M_p (\sqrt{r};f)\,, \qquad 0<r<1\,.
$$
Thus, the simple change of variable $\r^2=r$ in \eqref{eq-norm} yields
\begin{equation}
 \|f\|_{p,q,a} = \( \int_0^1 a q (1-r)^{a q-1} m_p^q (r) dr \)^{1/q}\,.
 \label{eq-simpl}
\end{equation}
This formula will be used frequently in the rest of this note. H\"older's inequality shows that $M_p (r;f)$, and hence also $m_p (r)$, is an increasing function of of $p$ for each fixed $r$.
\par
Also, by a well-known theorem of Hardy \cite[Chapter~1]{D}, $M_p (r;f)$  is an increasing function of $r$ for each fixed $p$. Moreover, it is strictly increasing whenever $f$ is not constant. This additional fact, although not explicit in \cite{D}, can be deduced immediately from the Hardy-Stein identity \cite[p. 174]{Po}. See also \cite{XZ} for a proof in the unit ball and some applications.
\subsection{Positive results on contractive inclusions}
 \label{subsec-pos-res}
In this subsection, we show that the inclusion under the conditions of Theorem~\ref{thm-ia1} is contractive in two large classes of cases.
\par
We first address the case when $p\ge u$, $q\le v$, and $a=b$.
\par
\begin{prop} \label{prop-aa}
Let $\infty\ge p\ge u>0$, $0<q\le v\le\infty$, $0<a<\infty$. Then \ $\|f\|_{u,v,a} \le \|f\|_{p,q,a}$.
\end{prop}
\begin{proof}
We first consider the case $v<\infty$. For $0\le R<1$, consider the function
$$
 \f(R) = \( \int_R^1 a q (1-r)^{a q-1} m_p^q (r) dr \)^{v/q} - \int_R^1 a v (1-r)^{a v-1} m_u^v (r) dr
$$
In order to prove the desired inequality, it suffices to show that $\f(0)\ge 0$. Next, it is enough to check that $\f$ is decreasing in $[0,1)$ since $\lim_{R\to 1^-} \f(R)=0$.
\par
We first note that, since $m_p(r)$ is an increasing function of $r$, we have
$$
 \int_R^1 a q (1-r)^{a q-1} m_p^q (r) dr \ge m_p^q (R) \int_R^1 a q (1-r)^{a q-1} dr = m_p^q (R) (1-R)^{a q}\,.
$$
Taking into account this inequality and the fact that $v/q-1\ge 0$ by our assumptions, the fundamental theorem of Calculus yields
\begin{eqnarray*}
 \f^\prime (R) &=& a v (1-R)^{a v-1} m_u^v (R) - a v (1-R)^{a q-1} m_p^q  (R) \( \int_R^1 a q (1-r)^{a q-1} m_p^q (r) dr \)^{v/q - 1}
\\
 & \le & a v (1-R)^{a v-1} m_u^v (R) - a v (1-R)^{a q-1} m_p^q (R) \cdot m_p^{v - q} (R) (1-R)^{a v - a q}
\\
 & = & a v (1-R)^{a v-1} m_u^v (R) - a v (1-R)^{a v-1} m_p^v (R)
 \le 0\,,
\end{eqnarray*}
since $m_u (R)\le m_p (R)$ for all $R\in [0,1)$ by the assumption $p\ge u$. This completes the proof.
\par\medskip
We now discuss the case $v=\infty$. There are two possibilities: $q=\infty$ and $q<\infty$.
\par
If $q=\infty$, the corresponding estimates are quite obvious. For every $\r\in [0,1)$ we have
$$
 (1-\r^2)^a M_u(\r;f) \le (1-\r^2)^a M_p(\r;f) \le \|f\|_{p,\infty,a}
$$
since $p\ge u$. This readily implies \ $\|f\|_{u,\infty,a} \le \|f\|_{p,\infty,a}$.
\par
For $0<q<\infty$, recalling that the integral means $M_p(r;f)$ are increasing in $p$ for a fixed $r$ and increasing in $r$ for a fixed $p$, for every $r\in [0,1)$ we obtain
\begin{eqnarray*}
 (1-r^2)^{aq} M_u^q(r;f) &\le & (1-r^2)^{aq} M_p^q(r;f)
\\
 &=& a q \int_r^1 2\r (1-\r^2)^{aq-1} d\r \cdot  M_p^q(r;f)
\\
 &\le & a q \int_r^1 2\r (1-\r^2)^{aq-1} M_p^q(\r;f) d\r
\\
 &\le & \|f\|_{p,q,a}^q \,.
\end{eqnarray*}
It follows from here that \ $\|f\|_{u,\infty,a} \le \|f\|_{p,q,a}$.
\end{proof}
\par
We now consider a different class of cases.
\begin{prop} \label{prop-ab}
Let $\infty\ge p\ge u>0$, $\infty>q\ge v>0$, $a q \le b v$. Then \ $\|f\|_{u,v,b} \le \|f\|_{p,q,a}$.
\end{prop}
\begin{proof}
It is useful to apply the change of variable $s=(1-r)^{b v}$. Applying, in the following order, H\"older's inequality (note that $q\ge v$ and we have a unit measure), the assumption that $a q\le b v$ and the fact that the function $F(t)=1-s^{1/t}$ decreases with $t$ for each fixed $s\in [0,1)$, recalling that $m_u (\r)\le m_p (\r)$ for each $\r\in [0,1)$ (since $p\ge u$ by assumption), and finally undoing the similar change of variable $s=(1-r)^{a q}$, we obtain
\begin{eqnarray*}
 \|f\|_{u,v,b}^q &=& \( \int_0^1 b v (1-r)^{b v -1} m_u^v (r) d r \)^{q/v} = \( \int_0^1 m_u^v (1 - s^{1/(b v)}) d s \)^{q/v}
\\
 & \le & \int_0^1 m_u^q (1 - s^{1/(b v)}) d s
 \le\int_0^1 m_p^q (1 - s^{1/(b v)}) d s
 \le \int_0^1 m_p^q (1 - s^{1/(a q)}) d s
\\
 &=& \int_0^1 a q (1-r)^{a q -1} m_p^q (r) d r
 = \|f\|_{p,q,a}^q \,.
\end{eqnarray*}
This proves the statement.
\end{proof}
\par
The following trivial corollary will be useful later as an auxiliary step. The case $q=\infty$ does not follow directly from Proposition~\ref{prop-ab} but is deduced rather easily.
\begin{cor} \label{cor-simple}
Let $0<p\le\infty$, $0<q\le\infty$, $0<a$, $b<\infty$, with $a\le b$. Then
$$
 \|f\|_{p,q,b} \le \|f\|_{p,q,a}\,.
$$
\end{cor}
\smallskip

\subsection{Counterexamples for the contractive property}
 \label{subsec-counterexs}
We recall that the Euler Beta function, defined in the usual way as
$$
 B(p,q) = \int_0^1 r^{p-1} (1-r)^{q-1} dr \,, \qquad p\,, q > 0\,,
$$
has the symmetry property that $B(p,q) = B(q,p)$. It also satisfies
the well-known identity
$$
 B(p,q) = \frac{\G(p) \G(q)}{\G(p+q)}\,,
$$
where
$$
 \G (p) = \int_0^\infty t^{p-1} e^{-t}\, dt\,, \qquad p>0\,,
$$
is the standard Euler Gamma function. The basic properties of these functions can be found in a number of texts, \textit{e.g.}, in \cite[Chapter~9]{Di}.
\par
Using some properties of the Euler integrals and appropriately chosen test functions, we will now show that, under the assumptions of Theorem~\ref{thm-ia1}, the inclusions between our mixed norm spaces can only be contractive in the cases covered in Subsection~\ref{subsec-pos-res}.
\par
\begin{prop} \label{prop-counterex}
Let $\infty\ge p \geq u>0$, $\infty\ge q > v>0$ and $0<a < b<\infty$. If $aq > bv$, then the inclusion $H(p, q, b) \subset H(u, v, b)$ is not contractive.
\end{prop}
\begin{proof}
It suffices to find a function $f \in H(p, q, a)$ such that
\[
 \|f\|_{p, q, a} < \|f\|_{u, v, b}.
\]			
We split the proof into two cases, depending on whether $q$ is finite or not.
\par\smallskip  	
In the case when $q = \infty$, it suffices to consider the function $f(z) = (1 + z^2)^a$. Then
\[
 (1-r^2)^a M_p(r,f) \leq (1 - r^2)^a M_\infty (r, f) \leq 1 = |f(0)| < \|f\|_{u, v, b}, \quad \forall r \in [0, 1).
\]
The last inequality follows from the fact mentioned earlier that $M_p(r;f)$ is a strictly increasing function whenever $f$ is not constant.
\par\smallskip
If $q<\infty$, without loss of generality we may assume that $p = \infty$ and $H(u, v, b)$ is actually a weighted Bergman space. The reason is that, due to Hölder's inequality, for every analytic function $f$ we have
\[
 \|f\|_{p, q, a} \leq \|f\|_{\infty, q, a} \quad \mbox{ and } \quad \|f\|_{A^{\min \{ u, v \}}_{bv - 1}} \leq  \|f\|_{u, v, b}\,,
\]
hence, choosing $w = \min \{u, v\}$, it suffices to find a function $f$ such that
\[
 \|f\|_{\infty, q, a} < \|f\|_{A^w_{bv - 1}}\,.
\]
To this end, let $\gamma \in \left(0, \frac{a}{2} \right)$ and consider the sequence of functions
\[
 f_n(z) = \dfrac{1}{(1 - z^{2n})^{2\gamma}}, \quad n \geq 1.
\]
Using repeatedly the basic property of the Gamma function: $\G (p+1) = p \G (p)$, the relationship between $B$ and $\G$, and the binomial expansion
$$
 (1-x)^{-c} = \sum_{k=0}^\infty \binom{-c}{k} (-1)^k x^k = \sum_{k=0}^\infty \frac{\G (c+k)}{k! \G (c)} x^k\,,
$$
a routine computation shows that
\[
 \|f_n\|_{\infty, q, a}^q = a q \int_0^1 (1-r)^{aq-1} (1-r^n)^{-2\g q}
 dr = \dfrac{\Gamma(aq+1)}{\Gamma(2\gamma q)} \sum_{k = 0}^\infty \dfrac{\Gamma(k + 2 \gamma q)}{k!} \dfrac{(nk)!}{\Gamma(nk + aq + 1)},
\]
%
%
The Euler-Gauss formula for the $\Gamma$ function \cite[\S~9.6]{Di} states that
\[
 \lim_{m \rightarrow \infty} \frac{(m-1)! m^c}{\Gamma (m + c)} = 1,
\]
for every $c>0$ (actually, this holds even for complex values). Hence, we can find a positive integer $m_1$ such that
$$
 \|f_n\|_{\infty, q, a}^q \le 1 + \dfrac{3\Gamma(aq+1)}{2\Gamma(2\gamma q)} \sum_{k = 1}^\infty \dfrac{\Gamma(k + 2 \gamma q)}{k!} \dfrac{1}{k^{aq}} \dfrac{1}{n^{aq}} = 1 + \dfrac{C_1(q, a, \gamma)}{n^{aq}}\,,
$$
for all $n \geq m_1$. On the other hand, there exists a positive integer $m_2$ such that
\begin{eqnarray*}
 \|f_n\|_{A^w_{bv-1}}^w & = & 1 + \dfrac{\Gamma(bv+1)}{\Gamma^2(\gamma w)} \sum_{k = 1}^\infty \dfrac{\Gamma^2(k+\gamma w)}{k!^2} \dfrac{(2nk)!}{\Gamma(2nk + bv +1)} \\
 & \geq & 1 + \dfrac{\Gamma(bv+1)}{2^{bv+1}\Gamma^2(\gamma w)} \sum_{k = 1}^\infty \dfrac{\Gamma^2(k+\gamma w)}{k!^2} \dfrac{1}{k^{bv}} \dfrac{1}{n^{bv}} \\
 & = & 1 + \dfrac{C_2(w, v, b, \gamma)}{n^{bv}}\,,
\end{eqnarray*}
for all $n \geq m_2$.
\par
Since $q>v\ge w$ and $\|f_n\|_{A^w_{bv-1}} \ge 1$, it follows that
$$
 n^{bv} \left(\|f_n\|_{A^w_{bv-1}}^q - \|f_n\|_{\infty, q, a}^q \right ) \ge n^{bv} \left(\|f_n\|_{A^w_{bv-1}}^w - \|f_n\|_{\infty, q, a}^q \right ) \ge C_2 - \dfrac{C_1}{n^{aq - bv}},
$$
if $n \geq \max \{ m_1, m_2 \}$. Therefore, in view of the assumption $aq> bv$, for $n$ large enough we have
\[
 \|f_n\|_{\infty, q, a} < \|f_n\|_{A^w_{bv-1}}\,,
\]
as claimed.
\end{proof}

\subsection{Main result}
 \label{subsec-main}
We now collect all the results obtained and summarize them in our main result below.
\par
\begin{thm} \label{thm-contr-ineq}
Let $q$ and $v$ be arbitrary with $0<q$, $v\le\infty$, and $0<a$, $b<\infty$, $\infty\ge p\ge u>0$, and either $a<b$ or $a=b$ and $q\le v$. Then the norm inequality for the corresponding inclusion $H(p,q,a)\subset H(u,v,b)$, asserted by Theorem~\ref{thm-ia1}, is contractive:
$$
 \|f\|_{u,v,b} \le \|f\|_{p,q,a}
$$
if and only if we have one of the following cases:
\par\smallskip
(1) $q\le v$ \ and \ $a\le b$, \quad or \quad (2) $q>v$ \ and \ $a q \le b v$.
\par\smallskip
In the special case of weighted Bergman spaces, this means that for $p\ge q$, the contractive inequality
$$
 \|f\|_{A^q_\b} \le \|f\|_{A^p_\a}
$$
holds if and only if $\a\le\b$.
\end{thm}
\begin{proof}
\par
To prove the result in the case $0<q\le v\le\infty$ and $a\le b$, it suffices to apply Corollary~\ref{cor-simple} and then Proposition~\ref{prop-aa} to obtain
$$
 \|f\|_{u,v,b} \le \|f\|_{u,v,a} \le \|f\|_{p,q,a}\,.
$$
The case $\infty>q>v$ and $a q \le b v$ is covered by Proposition~\ref{prop-ab}. It follows from Proposition~\ref{prop-counterex} that the inclusion is non-contractive in the remaining cases.
\par
For the special case of weighted Bergman spaces, recall that $A^p_\a = H(p,p,\frac{\a+1}{p})$.
\end{proof}
\par\smallskip
The following remark is in order. The main case of interest for contractive inequalities is that of weighted Bergman spaces; however, our results are stated in greater generality. It should be noted that this is not done just for the sake of generalizing. An inspection of various proofs shows that we have actually used genuine mixed norm spaces which are not weighted Bergman spaces in order to derive our results. Thus, the more general point of view has actually helped to obtain relatively simple proofs while at the same time yielding more general results.

\section{Some inequalities for the Beta function}
 \label{sec-ineq-Beta}
\par
In this section we formulate some consequences of the results of Subsection~\ref{subsec-pos-res} separately, in view of their possible independent interest. Namely, the contractive inequalities proved there  yield certain (seemingly new) inequalities for the Beta function. First, as a consequence of Proposition~\ref{prop-aa}, we obtain the following property.
\begin{cor} \label{cor-Beta1}
Let $a>0$ and let a positive integer $n$ be fixed. Then
$F(q) = (aq B(aq, nq+1))^\frac{1}{q}$ is a decreasing function of $q$.
\end{cor}
\begin{proof}
Apply Proposition~\ref{prop-aa} to the function $f(z)=z^{2n}$, observing that
\begin{equation}
 \|z^{2n}\|_{p,q,a} = \( \int_0^1 a q (1-r)^{a q-1} r^{nq} dr \)^{1/q} = \(aq B(aq, nq+1)\)^\frac{1}{q}\,.
 \label{eq-norm-monom}
\end{equation}
\end{proof}
\par
Next, from Proposition~\ref{prop-ab}, we obtain the following inequality which we have not been able to find in the literature. It does not seem to follow in the usual way from the convexity of the logarithm of the Gamma function and we have not been able to deduce it in an elementary way without essentially repeating the procedure employed in the proof of Proposition~\ref{prop-aa}.
\begin{cor} \label{cor-Beta2}
Let $x>1$, $y>0$, and $\d\ge 0$. Then
$$
 y^{n\delta} B\(x, y\)^{x-1+n\delta} \le
 B\(x+n\delta, y\)^{x-1}\,.
$$
for all positive integers $n$.
\end{cor}
\begin{proof}
We apply Proposition~\ref{prop-ab} to the function $f(z)=z^{2n}$, using again formula \eqref{eq-norm-monom}. This yields
$$
 \(bv B(bv, nv+1)\)^\frac{1}{v} \le \(aq B(aq, nq+1)\)^\frac{1}{q}\,.
$$
assuming that $\infty\ge p\ge u>0$, $\infty>q\ge v>0$, $a q \le b v$.
Let $x=nv+1$, $q=v+\d$ and choose $a$, $b>0$ in such a way that $aq=bv(=y)$. By the symmetry property of $B$, it follows that
$$
 \( y B(x, y) \)^{1/(x-1)} \le
 \( yB(x+n\delta, y) \)^{1/(x-1+n\d)}\,,
$$
and simplifications yield the desired inequality.
\end{proof}

\textsc{Acknowledgments}. The authors are partially supported by the grant PID2019-106870GB-I00 from MICINN, Spain. The first author is supported by the MU Predoctoral Fellowship FPU/00040.


\end{document}